\documentclass[11pt,reqno]{amsart}

\usepackage[T1]{fontenc}
\usepackage[utf8]{inputenc}
\usepackage[margin=1in]{geometry}
\usepackage{amsmath,amssymb,amsthm,amsfonts}
\usepackage{xcolor}
\usepackage{enumitem}
\usepackage{microtype}
\usepackage{indentfirst}
\usepackage[colorlinks=true,linkcolor={blue!75!black},citecolor={blue!80!black},urlcolor={blue!65!black}]{hyperref}

\numberwithin{equation}{section}

\theoremstyle{plain}
\newtheorem{theorem}{Theorem}[section]
\newtheorem{lemma}[theorem]{Lemma}

\newtheorem{problem}[theorem]{Problem}

\theoremstyle{definition}
\newtheorem{definition}[theorem]{Definition}

\newcommand{\N}{\mathbb{N}}

\begin{document}

\title{Gaps in Multiplicative Sidon Sets}

\author{Wouter van Doorn}
\address{Groningen, the Netherlands}
\email{wonterman1@hotmail.com}

\author{Pietro Monticone}
\address{Harmonic, London, United Kingdom}
\email{pietro.monticone@harmonic.fun}

\author{Quanyu Tang}
\address{School of Mathematics and Statistics, Xi'an Jiaotong University, Xi'an 710049, P. R. China}
\email{tang\_quanyu@163.com}

\subjclass[2020]{11B75, 11B05, 11B83, 11N05, 11N36, 05D15, 68V20}
\keywords{multiplicative Sidon set, prime gaps, Baker--Harman--Pintz,
Laishram--Murty, Hall's marriage theorem,
formal verification, automated theorem proving}

\begin{abstract}
For a positive integer $n$, let $g(n)$ denote the infimum of all real numbers
$L$ such that there exists a multiplicative Sidon set
$A\subseteq\{1,2,\dots,n\}$ that intersects every interval
$[x,x+L]\subseteq[1,n]$. S\'ark\"ozy asked for estimates on $g(n)$, and he in
particular asked whether one has $g(n)\le\sqrt n$ for every $n\in\N$. We first
show that this estimate does indeed hold, with a proof that was autonomously
discovered and formally verified in Lean by
\textnormal{\href{https://aristotle.harmonic.fun/}{Aristotle}}. Next, we improve the
upper bound further and, with $\rho = \frac{13-\sqrt{69}}{10} < 0.47$, prove
that $g(n)\ll_{\varepsilon} n^{\rho+\varepsilon}$ for every $\varepsilon > 0$.
\end{abstract}

\maketitle

\section{Introduction}

With $\N$ the set of positive integers, we say that a set $A\subseteq\N$ is a
\emph{multiplicative Sidon set} if all products $aa'$ with $a,a'\in A$ and $a\le
a'$ are distinct. Naturally, the set of prime numbers is such a set, from which
we deduce that there exist multiplicative Sidon sets in $[n]:=\{1,2,\dots,n\}$
with at least $\pi(n)$ integers. Erd\H{o}s~\cite{Erdos1938, Erdos1968} proved
that this is the right asymptotic and determined the correct order of growth of
the error term. More precisely, he showed that the maximum size of a
multiplicative Sidon set contained in $[n]$ is
\begin{equation} \label{eq:erdasymp}
\pi(n)+\Theta\!\left(\frac{n^{3/4}}{(\log n)^{3/2}}\right).
\end{equation}
Thus the extremal-size problem is well-understood, even though it is still open
what the optimal implied constant is; see e.g. \cite{Bloom425}. On the other
hand, much less is known about gap sizes in multiplicative Sidon sets. In this
regard, S\'ark\"ozy~\cite[Problem~34]{Sarkozy2001} asked the following questions
in his 2001 list of unsolved problems.

\begin{problem}[S\'ark\"ozy~\cite{Sarkozy2001}]\label{prob:34}
How small can one make the maximal gap between the consecutive elements of a
multiplicative Sidon set selected from $\{1,2,\dots,n\}$? Does there exist for
all $n$ a multiplicative Sidon set $A\subseteq\{1,2,\dots,n\}$ so that
$A\cap [m,m+\sqrt n]\neq\varnothing$ for all $1\le m<m+\sqrt n\le n$?
\end{problem}

As S\'ark\"ozy~\cite{Sarkozy2001} points out, if one could prove that the
sequence of prime numbers $p_1, p_2, \ldots$ satisfies the estimate $p_{i+1}-p_i
\le p_i^{1/2}$, then an affirmative answer to the second question would follow.
However, even under the assumption of the Riemann Hypothesis such an estimate is
not known.

To quantify the gap size question, let us introduce the following function.

\begin{definition}
For a positive integer $n$, define $g(n)$ as the infimum of all real numbers
$L\ge 0$ for which there exists a multiplicative Sidon set $A\subseteq[n]$
satisfying $A\cap [x,x+L]\neq\varnothing$ for every $x\in\mathbb{R}$ with
$[x,x+L]\subseteq[1,n]$.
\end{definition}

By combining the prime number theorem, Erd\H{o}s' estimate \eqref{eq:erdasymp},
and the pigeonhole principle, the lower bound $g(n)\ge(1+o(1))\log n$ quickly
follows.

In the other direction, the best known upper bound on $g(n)$ obtained purely
from prime gaps is $g(n) \le n^{21/40}$ for all large enough $n$, by a theorem
of Baker, Harman and Pintz~\cite{BakerHarmanPintz2001}.

In this paper we manage to lower the exponent and reach the square-root barrier,
answering S\'ark\"ozy's second question affirmatively.

\begin{theorem}\label{thm:main-all}
For all $n\in\N$ one has
\[
g(n)\le\lfloor\sqrt n\rfloor.
\]
\end{theorem}

By making use of two external inputs it is possible to do even better, however.
The first is the aforementioned theorem of Baker, Harman and Pintz
\cite{BakerHarmanPintz2001} that every interval $(x - x^{21/40}, x]$ contains a
prime for all sufficiently large $x$. The second is a lower bound of
Laishram--Murty for
\[
\sum_{1 \le m\le x^\beta}\left\{\pi\!\left(\frac{x+x^\alpha}{m}\right)
-\pi\!\left(\frac{x}{m}\right)\right\}
\]
with suitable parameters $\alpha,\beta$~\cite[Eq.~(12)]{LaishramMurty2012}. This
gives the following asymptotic improvement.

\begin{theorem}\label{thm:main-rho}
Let
\[
\rho:=\frac{13-\sqrt{69}}{10}\approx 0.46934.
\]
Then for every $\varepsilon>0$ there exists a constant $C_\varepsilon>0$ such
that
\[
g(n)\le C_\varepsilon\, n^{\rho+\varepsilon}
\]
for all $n\in\N$.
\end{theorem}

The proof of Theorem~\ref{thm:main-all} was autonomously discovered and
formally verified in Lean by \emph{Aristotle}. The proof of
Theorem~\ref{thm:main-rho} was found by the authors and subsequently
formalized in Lean, also by \emph{Aristotle}. Further information on the
discovery and verification workflow can be found in the
\hyperref[sec:appendix]{appendix}.

\section{An elementary construction}\label{sec:elem}

In this section we prove Theorem~\ref{thm:main-all}.
\begin{proof}[Proof of Theorem~\ref{thm:main-all}]
We put
\[
q:=\lfloor\sqrt n\rfloor,\qquad
A:=\{a\le n:a\equiv 1\pmod q\}.
\]

It is clear that $A$ intersects every interval of length $q$, so it is
sufficient to show that $A$ is a multiplicative Sidon set. Since $1+q(q+2)>n$,
every element of $A$ has the form $1+qi$ with $0\le i\le q+1$. Thus it suffices
to prove that, whenever $i, j, k, \ell$ satisfy
\[
0\le i\le j\le q+1 \qquad \text{and} \qquad 0\le k\le \ell\le q+1,
\]
then the equality
\[
(1+qi)(1+qj)=(1+qk)(1+q\ell)
\]
implies $i=k$ and $j=\ell$.

Expanding the brackets on both sides gives
\[
1+q(i+j)+q^2ij=1+q(k+\ell)+q^2k\ell.
\]
After subtracting $1$ and dividing by $q$, we obtain
\begin{equation} \label{eq:delta}
(i+j-k-\ell)+q(ij-k\ell)=0.
\end{equation}
With $\Delta:=i+j-k-\ell$, we then see $q\mid\Delta$, so either $\Delta=0$ or
$|\Delta|\ge q$.

If $\Delta=0$, then $ij=k\ell$ by equation \eqref{eq:delta}, so the pairs
$(i,j)$ and $(k,\ell)$ have the same sum and the same product. They are
therefore the two roots of the same quadratic equation $X^2-(i+j)X+ij=0$, and
thus $i=k$ and $j=\ell$. It is therefore sufficient to show that the assumption
$|\Delta| \ge q$ leads to a contradiction, and by symmetry we may further assume
$\Delta \ge q$.

Writing $s:=k+\ell$, we see that equation \eqref{eq:delta} implies $s > 0$ and
\begin{equation} \label{eq:ijkls2}
ij < k\ell \le \left\lfloor \frac{s^2}{4} \right\rfloor.
\end{equation}
From $j\le q+1$ and $i+j=s+\Delta$ we get
\[
i=s+\Delta-j\ge s-1.
\]
Thus $i\in[s-1,q+1]$ and
\[
ij=i(s+\Delta-i)\ge i(s+q-i).
\]
The latter product is a concave quadratic polynomial in $i$, so its minimum on
$[s-1,q+1]$ is attained at an endpoint, which gives $ij \ge (s-1)(q+1)$. Now,
$\Delta \ge q$ implies
\[
s = i+j - \Delta \le 2q+2 - q = q+2.
\]
We therefore obtain
\[
ij \ge (s-1)(q+1) \ge (s-1)^2 \ge \left\lfloor \frac{s^2}{4} \right\rfloor,
\]
contradicting inequality \eqref{eq:ijkls2}. This completes the proof.
\end{proof}

\section{Preliminaries}\label{sec:prelim}

\subsection{A multiplicative Sidon criterion}

The fact that the set of primes is a multiplicative Sidon set can be seen as a
special case of the following more general criterion.

\begin{lemma}\label{lem:private-prime}
Let $J\ge 1$, and let $A$ be a set of integers such that every $a_i\in A$ can be
written in the form $a_i=m_ip_i$, where $p_i$ is a prime, $1 \le m_i \le J <
p_i$, and the primes $p_i$ are pairwise distinct as $a_i$ ranges over $A$. Then
$A$ is a multiplicative Sidon set.
\end{lemma}

\begin{proof}
Assume $a_1a_2=a_3a_4$, and suppose by contradiction that $a_1 \notin \{a_3,
a_4\}$. Write $a_i=m_ip_i$ with $1\le m_i\le J<p_i$ and $p_1$ different from
$p_3$ and $p_4$. Since $p_1>J\ge m_3,m_4$, the prime $p_1$ cannot divide
$m_3m_4$. Hence $p_1$ must divide $p_3p_4$, so $p_1=p_3$ or $p_1=p_4$,
contradicting the assumption that these primes are distinct. Hence, $a_1 \in
\{a_3, a_4\}$ and cancelling the common factor gives $\{a_1, a_2\} = \{a_3,
a_4\}$.
\end{proof}

\subsection{Primes in short intervals}

As we alluded to before, we shall use the following result of Baker, Harman and
Pintz~\cite[Theorem~1]{BakerHarmanPintz2001}.

\begin{lemma}[\cite{BakerHarmanPintz2001}]\label{lem:BHP}
For all sufficiently large real numbers $x$, the interval $(x-x^{21/40},x]$
contains a prime.
\end{lemma}

We will also use the following consequence
of~\cite[Eq.~(12)]{LaishramMurty2012}.

\begin{lemma}\label{lem:LM-param}
Let $0<\beta<\alpha<\tfrac12$, and let $\delta>0$ satisfy
$3\alpha-4/3<\delta<(5\alpha-2)/3$. Set
\[
\eta:=1-\alpha-\frac{1-\alpha^2-\beta(2-\beta)}{\delta}.
\]
Then for every real number $c_0<\eta$, one has
\begin{equation}\label{eq:LM-param}
\sum_{1\le m\le x^\beta}\left(\pi\!\left(\frac{x+x^\alpha}{m}\right)
-\pi\!\left(\frac{x}{m}\right)\right)
\ge c_0\,x^\alpha
\end{equation}
for all sufficiently large $x$. Moreover, the primes counted on the left-hand
side of~\eqref{eq:LM-param} are distinct. Consequently, for all sufficiently
large $x$, there exist at least $c_0\,x^\alpha$ distinct primes $p$ for which
$mp\in(x,x+x^\alpha]$ for some integer $m\le x^\beta$.
\end{lemma}

\begin{proof}
Since $0<\beta<\alpha<\tfrac12$ and $3\alpha-4/3<\delta<(5\alpha-2)/3$, the
estimate stated in~\cite[Eq.~(12)]{LaishramMurty2012} applies with our
parameters $\alpha,\beta,\delta$. More precisely, for every $\varepsilon'>0$
and all sufficiently large $x$,
\[
\sum_{1\le m\le x^\beta}\left(\pi\!\left(\frac{x+x^\alpha}{m}\right)
-\pi\!\left(\frac{x}{m}\right)\right)
\ge\left(1-\alpha-\varepsilon'-\frac{1-\alpha^2-\beta(2-\beta)}{\delta}\right)
x^\alpha.
\]
Since $c_0<\eta$, we may choose $\varepsilon'>0$ so small that
\[
1-\alpha-\varepsilon'-\frac{1-\alpha^2-\beta(2-\beta)}{\delta}>c_0,
\]
proving~\eqref{eq:LM-param}. It remains to prove the distinctness statement,
which is essentially equivalent to the inequality $\frac{x+x^\alpha}{m+1} \le
\frac{x}{m}$ for all $1\le m\le x^\beta$. Multiplying this inequality by
$m(m+1)$ gives $mx^\alpha \le x$, which holds for all $m\le x^\beta$ since
$\alpha+\beta<1$. The final assertion follows immediately.
\end{proof}

\subsection{A weighted Hall-type lemma}

We shall also use the following weighted variant of Hall's theorem.

\begin{lemma}\label{lem:weighted-Hall}
Let $G=(\mathcal{L},\mathcal{R},\mathcal{E})$ be a finite bipartite graph, and
suppose that each edge $e\in \mathcal{E}$ is assigned a non-negative weight $w(e)$.
Assume that there exists a real number $L_0>0$ such that
\[
\sum_{\substack{e\in \mathcal{E}\\ e\ni u}}w(e)\ge L_0
\qquad\text{for every }u\in \mathcal{L},
\]
and
\[
\sum_{\substack{e\in \mathcal{E}\\ e\ni v}}w(e)\le L_0
\qquad\text{for every }v\in \mathcal{R}.
\]
Then $G$ has a matching covering all vertices of $\mathcal{L}$.
\end{lemma}

\begin{proof}
Let $S\subseteq \mathcal{L}$, and let $N(S)\subseteq \mathcal{R}$ be its
neighborhood. Then
\[
L_0|S|\le\sum_{u\in S}\sum_{\substack{e\in \mathcal{E}\\ e\ni u}}w(e).
\]
Every edge counted on the right-hand side has its $\mathcal{L}$-endpoint in $S$,
hence its $\mathcal{R}$-endpoint lies in $N(S)$. Therefore
\[
\sum_{u\in S}\sum_{\substack{e\in \mathcal{E}\\ e\ni u}}w(e)
\le\sum_{v\in N(S)}\sum_{\substack{e\in \mathcal{E}\\ e\ni v}}w(e)
\le L_0|N(S)|.
\]
Thus $|S|\le|N(S)|$. So Hall's condition holds, and Hall's theorem yields a
matching covering $\mathcal{L}$.
\end{proof}

\section{Proof of the power-saving bound}\label{sec:rho}

\begin{proof}[Proof of Theorem~\ref{thm:main-rho}]
Since $\rho = \frac{13-\sqrt{69}}{10} < \frac{19}{40}$, without loss of
generality we may assume that $0<\varepsilon<\tfrac{19}{40}-\rho$. Choose
$\alpha$ so that $\rho<\alpha<\rho+\varepsilon$. Then $\alpha < \frac{19}{40}$
and $5\alpha^2-13\alpha+5<0$, where the last inequality is equivalent to
\[
1-\alpha- \frac{1-2\alpha}{(5\alpha-2)/3} > 0.
\]
We now define
\[
F(\beta,\delta):=1-\alpha-\frac{1-\alpha^2-\beta(2-\beta)}{\delta}.
\]
Then $F\bigl(\alpha,(5\alpha-2)/3\bigr)>0$, while $\alpha < \rho + \varepsilon
<1/2$ implies $3\alpha-4/3<(5\alpha-2)/3$. By continuity of $F$, we may
therefore choose real numbers $\beta,\delta$ such that
\[
0<\beta<\alpha<\tfrac12,
\qquad
3\alpha-\tfrac43<\delta<\frac{5\alpha-2}{3},
\]
and $\eta:=F(\beta,\delta)>0$. Applying Lemma~\ref{lem:LM-param} with
$c_0:=\eta/2>0$ then gives
\begin{equation}\label{eq:LM-main-rho}
\sum_{1\le m\le x^\beta}\left(\pi\!\left(\frac{x+x^\alpha}{m}\right)
-\pi\!\left(\frac{x}{m}\right)\right)\ge c_0\,x^\alpha
\end{equation}
for all sufficiently large $x$.

Now let $n$ be sufficiently large and define
\[
H:=\lceil 2n^\alpha\rceil,\qquad
J:=\lfloor n^\beta\rfloor,\qquad
T:=\lfloor n/H\rfloor, \qquad
t:= \lfloor H^{19/21} \rfloor.
\]
With these definitions we further define the intervals
\[
B_i:=\big(iH, (i+1)H\big] \qquad (1 \le i < T).
\]

By Lemma~\ref{lem:private-prime} it is sufficient to find integers $m_1, \ldots,
m_{T-1} \le J$ and distinct primes $p_1, \ldots, p_{T-1} > J$ such that $m_ip_i
\in B_i$ for all $1 \le i < T$. Indeed, the first element $m_1p_1 \in B_1$ is
then at most $2H$, the difference between two consecutive elements is smaller
than $2H$, and the largest element $m_{T-1}p_{T-1} \in B_{T-1}$ is larger than
$n - 2H$. Assuming such $m_i$ and $p_i$ are chosen, with $C_\varepsilon = 5$ we
would then get $g(n) \le 2H \le C_\varepsilon n^{\rho+\varepsilon}$, and we can
cover the finitely many remaining values of $n$ by possibly enlarging
$C_\varepsilon$.

Now, for $i < t$, the largest element of $B_i$ is $(i+1)H$, while the length of
$B_i$ is
\[
H = H^{19/40} H^{21/40} \ge (tH)^{21/40} \ge ((i+1)H)^{21/40}.
\]
Lemma~\ref{lem:BHP} therefore implies that for every $i < t$, $B_i$ contains a
prime $p_i$. Moreover, for all such $i$ we have
\[
p_i \ge p_1 > H > J.
\]
We therefore choose $m_i = 1$ for all $1 \le i < t$.

As for the intervals $B_i$ with $t \le i < T$, with $x_i := iH$ the left
endpoint of $B_i$, define
\[
\mathcal{R}_i :=\bigl\{p : p \notin \{p_1, \ldots, p_{t-1}\} \text{ is prime and }mp \in B_i \text{ for some positive integer }m\le x_i^{\beta} \bigr\}.
\]
For $p \in \mathcal{R}_i$ we note that we also have
\[
p > \frac{x_i}{x_i^{\beta}} \ge x_t^{1 - \beta} = (tH)^{1 - \beta} > (n^{40\alpha/21})^{21/40} = n^{\alpha} > J.
\]
We now consider the bipartite graph with left vertex set $\mathcal{L} :=
\{B_{t}, \ldots, B_{T-1}\}$ and right vertex set $\mathcal{R} := \bigcup_{i=t}^{T-1} \mathcal{R}_i$. We then join a vertex $B_i \in \mathcal{L}$ to a prime
$p \in \mathcal{R}_i$, and we define the weight of such an edge $(B_i, p)$ by
\[
w(B_i,p):=\sum_{\substack{1\le m\le x_i^{\beta} \\ mp\in B_i }}\frac{1}{m}.
\]
As the intervals $B_i$ are disjoint, every fixed $m$ occurs in the summation on
the right-hand side for at most one index $i$. Hence, for all $p \in
\mathcal{R}$ we have
\[
\sum_{B_i \in \mathcal{L}} w(B_i,p) \le \sum_{1 \le m \le n} \frac{1}{m}.
\]
We therefore deduce by Lemma~\ref{lem:weighted-Hall} that it now suffices to
show that for all $t \le i < T$ we have
\begin{equation}\label{eq:row-after-delete}
\sum_{p\in\mathcal R}w(B_i,p) \ge \sum_{1 \le m \le n} \frac{1}{m}.
\end{equation}
Indeed, this would imply the existence of a matching, which gives us for all $t
\le i < T$ a prime $p_i > J$ such that $m_ip_i \in B_i$ for some $m_i \le
\left\lfloor x_i^{\beta} \right\rfloor \le \left\lfloor n^{\beta} \right\rfloor
= J$.

In order to show \eqref{eq:row-after-delete}, fix $i$ and $m$. There are at most
two primes $p\in\{p_1,\ldots,p_{t-1}\}$ such that $mp\in B_i$. Indeed, if
$p_j,p_k,p_\ell$ were three such primes with $1\le j<k<\ell<t$, then $\ell\ge
j+2$, and since $p_j\in B_j$ and $p_\ell\in B_\ell$ we have $p_\ell-p_j>H$. Thus
\[
mp_\ell-mp_j=m(p_\ell-p_j)>mH\ge H,
\]
contradicting the fact that both $mp_j$ and $mp_\ell$ lie in the interval $B_i$,
which has length $H$. Secondly, as $x_i^{\alpha} < n^{\alpha} < H$, the
interval $(x_i, x_i + x_i^{\alpha}]$ is contained in $B_i$. Hence, applying
equation \eqref{eq:LM-main-rho} gives
\begin{align*}
\sum_{p\in\mathcal{R}}w(B_i,p) &\ge \sum_{1 \le m\le x_i^\beta}\frac{1}{m}\left(\pi\!\left(\frac{x_i+x_i^\alpha}{m}\right)
-\pi\!\left(\frac{x_i}{m}\right)\right) - \sum_{\substack{1 \le m\le x_i^{\beta} \\ mp\in B_i \\ p \in \{p_1, \ldots, p_{t-1}\}}}\frac{1}{m} \\
&\ge c_0\,x_i^{\alpha - \beta} - 2\sum_{1 \le m \le n} \frac{1}{m} \\
&\ge c_0\,n^{\alpha(\alpha - \beta)} - 2\sum_{1 \le m \le n} \frac{1}{m} \\
&> 3\log(3n) - 2\log(3n) \\
&> \sum_{1 \le m \le n} \frac{1}{m}. \qedhere
\end{align*}
\end{proof}

\newpage

\section*{Appendix: Formal Discovery and Verification with Aristotle}
\label{sec:appendix}

The proof of Theorem~\ref{thm:main-all} was autonomously discovered and formally
verified in Lean by \emph{Aristotle}, a formal reasoning agent developed by
\textnormal{\href{https://www.harmonic.fun/}{Harmonic}}~\cite{Achim2025} and
publicly available for free at
\href{https://aristotle.harmonic.fun/}{\nolinkurl{aristotle.harmonic.fun}}.

We first asked Aristotle to investigate the second subquestion in
Problem~\ref{prob:34}. It autonomously produced a proof of
Theorem~\ref{thm:main-all}, together with a Lean formalization. After that, the
human authors studied the first subquestion in Problem~\ref{prob:34} more
carefully and managed to prove the stronger conclusion stated in
Theorem~\ref{thm:main-rho}. We then asked Aristotle to formalize the latter
proof, under the assumption of the results by Baker--Harman--Pintz and
Laishram--Murty that we used. Aristotle successfully completed this second
(conditional) formalization as well.

The full Lean file combining both formalizations can be inspected interactively
in the
\href{https://live.lean-lang.org/#project=mathlib-v4.28.0&url=https://gist.githubusercontent.com/pitmonticone/f02fc29ec8d16dc62cb5096749d4ccab/raw/0193455d7fef0bc4b0682699ea199e03ec06c1c9/Sarkozy_UPNT_34.lean}%
{Lean~4 Web Editor}.

\end{document}